\documentclass[a4paper,11pt]{amsart}
\usepackage{lmodern}
\usepackage{amsmath, amsthm, amssymb, mathrsfs, longtable, multirow}
\usepackage[inline]{enumitem}
\usepackage{graphicx}

\title%
[Spherical inversion 
for a small $K$-type on $G_2$]
{
Spherical inversion 
for a small $K$-type on the split real  Lie group of type $G_2$
}

\author{Hiroshi \textsc{Oda}}
\address{Faculty of Engineering, Takushoku University,
815-1 Tatemachi, Hachioji, Tokyo 193-0985, Japan}
\email{hoda@la.takushoku-u.ac.jp}
\author{Nobukazu \textsc{Shimeno}}
\address{School of Science \& Technology, Kwansei Gakuin University, 
2-1 Gakuen, Sanda, Hyogo 669-1337, Japan}
\email{shimeno@kwansei.ac.jp}

\subjclass[2010]{22E45, 43A90}
\keywords{small $K$-type, spherical function, spherical transform}

\dedicatory{Dedicated to Professor~Toshio~Oshima on the~occasion of his~70-th~birthday} 
\date{}

\newcommand{\bbC}{{\mathbb C}}

\newcommand{\bsm}{\boldsymbol m}

\newcommand{\reg}{{\mathrm{reg}}}

\DeclareMathOperator{\pr}{pr}

\DeclareMathOperator{\SU}{SU}

\newcommand{\simarrow}{\xrightarrow{\smash[b]{\lower 0.7ex\hbox{$\sim$}}}}

\DeclareMathOperator*{\Res}{Res}

\theoremstyle{plain}
 \newtheorem{thm}{Theorem}[section]
 \newtheorem{cor}[thm]{Corollary}
 \newtheorem{lem}[thm]{Lemma}
 
\theoremstyle{definition}

\theoremstyle{remark}

\calclayout

\allowdisplaybreaks[3]

\begin{document}
\begin{abstract}
We give an explicit formula for the Harish-Chandra $c$-function for a small $K$-type  
on a split  real  Lie group of type $G_2$. As an application we give an explicit formula 
for spherical inversion for this small $K$-type. 
\end{abstract}

\maketitle

\section{Introduction}
Harmonic analysis on a Riemannian symmetric space $G/K$ of the noncompact type 
is by now well developed (cf. \cite{Hel2}). 
A  natural extension is 
to study harmonic analysis on homogeneous vector bundles over $G/K$. 
One of fundamental problems 
in harmonic analysis is to establish the Plancherel theorem. 
Harish-Chandra establishes a general theory of 
 the Eisenstein integrals and 
the Plancherel theorem for noncompact real semisimple Lie groups (cf. \cite{HC76,Kn0,Wal,War}). 
The Plancherel theorem on a homogeneous vector bundle over $G/K$ 
associated with an irreducible representation $\pi$ of $K$ follows from 
Harish-Chandra's result by restricting the Plancherel measure to 
$K$-finite functions of type $\pi$. But 
it is a highly nontrivial and important problem to determine the Plancherel measure 
on the associated vector bundle as 
explicitly as in the case of the trivial $K$-type. 
There are several studies in this direction (cf. \cite{Camporesi2, CP, vDP, 
FJ, Hec, OS, SPlancherel, Shyperbolic}). 

In our previous paper \cite{OS}, we study elementary 
spherical functions on $G$ with a 
small $K$-type $\pi$ (in the sense of Wallach \cite[\S 11.3]{Wal}). 
Namely, we identify elementary spherical functions with the Heckman-Opdam 
hypergeometric function (cf. \cite{Hec, Op:lecture}) and apply the 
inversion formula and the Plancherel formula 
for the hypergeometric Fourier transform (\cite{Op:Cherednik}) 
to obtain the inversion formula and the Plancherel formula for the 
$\pi$-spherical transform. But there is an exception in \cite{OS}. 
Namely, for a certain small $K$-type of a noncompact Lie group 
of type $G_2$, elementary spherical functions can not be expressed by 
the Heckman-Opdam hypergeometric function. 

In this paper we give a complete treatment of harmonic analysis of 
$\pi$-spherical transform for each small $K$-type $\pi$ of  
 $G_2$. Namely, we give an explicit formula for the Harish-Chandra $c$-function 
$c^{\pi}(\lambda)$ 
and determine the Plancherel measure explicitly. The most continuous part 
of the Plancherel measure
 is $|c^\pi(\lambda)|^{-2}d\lambda$ on $\sqrt{-1}\mathfrak{a}^*$ and the other 
 spectra with supports of lower dimensions are given explicitly by 
 using residue calculus.  
As indicated by Oshima \cite{Oshima81a} and as was done for one-dimensional $K$-types 
by the second author \cite{SPlancherel}, we could prove the inversion formula for the 
$\pi$-spherical tranform in the case of $G_2$ by 
 extending Rosenberg's method of a proof of the inversion formula in the case of 
 the trivial $K$-type (\cite{R}). Instead of doing this, 
 we utilize general results on the Plancherel theorem and residue calculus on 
 $G$ due to Harish-Chandra and Arthur (cf. \cite{HC76, A, Kn0,Wal}) 
 and devote ourselves to the determination of the Plancherel measure. 

This paper is organized as follows. In Section~2 we give general results 
for elementary $\pi$-spherical functions, the 
Harish-Chandra $c$-function, the inversion formula for $\pi$-spherical transform 
with respect to a small $K$-type $\pi$ 
on a noncompact real semisimple Lie group of finite center. 
 
 In Section~3 we study the case of $G_2$. We give an  explicit formula of 
the $c$-function (Theorem~\ref{thm:cfg2}), 
the  inversion formula, and the Plancherel formula (Theorem~\ref{thm:main}, 
 Corollary~\ref{cor:main}) for each small $K$-type. 
In particular, they cover the small $K$-type that is not treated in \cite{OS}. 

\section{Elementary spherical functions 
for small $K$-types}

\subsection{Notation}
Let $\mathbb{N}$ denote the set of the nonnegative integers. 
Let $G$ be a non-compact connected real semisimple Lie group of 
finite center and $K$ a maximal compact subgroup of $G$. 
Let $e$ denote the identity element of $G$. 
Lie algebras of Lie groups $G,\,K$, etc. are denoted by the corresponding German letter  $\mathfrak{g},\,\mathfrak{k}$, etc.  
Let $\mathfrak{g}=\mathfrak{k}+\mathfrak{p}$ be 
the Cartan decomposition and $\mathfrak{a}$ a maximal abelian subspace 
of $\mathfrak{p}$. Let $\varSigma$ denote the root system for $(\mathfrak{g},\mathfrak{a})$. 
For $\alpha\in\varSigma$, let $\mathfrak{g}_\alpha$ denote the corresponding 
root space and $\boldsymbol{m}_\alpha=\dim \mathfrak{g}_\alpha$. 
Fix a 
positive system $\varSigma^+\subset \varSigma$ and let 
$\varPi=\{\alpha_1,\dots,\alpha_r\}$ denote the set of simple roots in $\varSigma^+$. 
Define $\mathfrak{n}=\sum_{\alpha\in\varSigma^+}\mathfrak{g}_\alpha$ and 
$N=\exp \mathfrak{n}$. Then we have the Iwasawa decomposition $G=K\exp\mathfrak{a}\,N$. 
Define $\rho=\frac12\sum_{\alpha\in\varSigma^+}\boldsymbol{m}_\alpha \alpha$. 

Let $W$ denote the Weyl group of $\varSigma$ and 
 $s_i$ the reflection across $\alpha_i^\perp$ ($1\leq i\leq r$). 
We have $W\simeq M'/M$, where $M'$ (resp. $M$) is the normalizer (resp. centralizer) 
of $\mathfrak{a}$ in $K$. 

Define
\begin{align*}
& \mathfrak{a}_+
=\{H\in\mathfrak{a}\,|\,\alpha(H)>0 \text{ for all } \alpha\in\varSigma^+\},
\\
& \mathfrak{a}_\text{reg}
=\{H\in\mathfrak{a}\,|\,\alpha(H)\not=0 \text{ for all } \alpha\in\varSigma^+\}.
\end{align*}
We have the Cartan decomposition $G=K\exp \overline{\mathfrak{a}_+}\,K$. 

Let $\langle\,\,,\,\,\rangle$ denote the inner product on $\mathfrak{a}^*$ 
induced by the Killing form on $\mathfrak{g}$ and $||\,\,||$ the corresponding norm. 
Define
\[
\mathfrak{a}_+^*
=\{\lambda\in\mathfrak{a}^*\,|\,\langle\lambda,\alpha\rangle>0 \text{ for all } \alpha\in \varSigma^+\}.
\]

\subsection{Elementary $\pi$-spherical function}

In this subsection, we review elementary $\pi$-spherical functions for small $K$-types 
according to \cite{OS}. 

Let $(\pi,V)$ be a small $K$-type, that is, $\pi|_M$ is irreducible. 
We call an $\text{End}_\mathbb{C} V$-valued function $f$ on $G$ satisfying 
\[
f(k_1gk_2)=\pi(k_2^{-1})f(g)\pi(k_1^{-1})\quad (k_1,\,k_2\in K,\,g\in G)
\]
a $\pi$-spherical function. 

Let $\boldsymbol{D}^\pi$ denote the algebra of the invariant differential operators on the homogeneous 
vector bundle over $G/K$ associated with $\pi$. 
Let $U(\mathfrak{g}_\mathbb{C})$ denote the universal enveloping algebra of 
$\mathfrak{g}_\mathbb{C}=\mathfrak{g}\otimes_\mathbb{R}\mathbb{C}$ and 
$U(\mathfrak{g}_\mathbb{C})^K$ the set of the 
$\text{Ad}(K)$-invariant elements in $U(\mathfrak{g}_\mathbb{C})$. 
Let $J_{\pi^*}=\text{ker}\,\pi^*$ in $U(\mathfrak{k}_\mathbb{C})$. 
We have 
\[
\boldsymbol{D}^\pi\simeq U(\mathfrak{g}_\mathbb{C})^K/
U(\mathfrak{g}_\mathbb{C})^K\cap U(\mathfrak{g}_\mathbb{C})J_{\pi^*}.
\] 

Let $S(\mathfrak{a}_\mathbb{C})$ denote the symmetric algebra 
of $\mathfrak{a}_\mathbb{C}=\mathfrak{a}\otimes_\mathbb{R}\mathbb{C}$ and 
$S(\mathfrak{a}_\mathbb{C})^W$ the set of the $W$-invariant elements in $S(\mathfrak{a}_\mathbb{C})$. 
There exists an algebra homomorphism 
\[
\gamma^\pi:U(\mathfrak{g}_\mathbb{C})^K\rightarrow S(\mathfrak{a}_\mathbb{C})^W
\]
with the kernel $U(\mathfrak{g}_\mathbb{C})^K\cap U(\mathfrak{g}_\mathbb{C})
J_{\pi^*}$ (cf. \cite[Lemma~11.3.2, Lemma~11.3.3]{Wal}). 
Notice that the homomorphism $\gamma^\pi$ is independent of the choice of 
$\varSigma^+$. 
Thus we have the generalized Harish-Chandra isomorphism 
$\gamma^\pi:\boldsymbol{D}^\pi\simarrow S(\mathfrak{a}_\mathbb{C})^W$. 
Therefore, any algebra homomorphism 
from $\boldsymbol{D}^\pi$ to $\mathbb{C}$ is of the form $D\mapsto \gamma^\pi(D)(\lambda)\,\,(D\in\boldsymbol{D}^\pi)$ 
for some $\lambda\in \lower0.8ex\hbox{$W$}\backslash \mathfrak a_\bbC^*$.

For  $\lambda\in \lower0.8ex\hbox{$W$}\backslash \mathfrak a_\bbC^*$ 
there exists a unique 
smooth $\pi$-spherical 
function $f=\phi_\lambda^\pi$ satisfying $f(e)=\text{id}_V$ and 
$Df=\gamma^\pi(D)(\lambda)f\,\,\,(D\in\boldsymbol{D}^\pi)$ (cf. \cite[Theorem~1.4]{OS}). 
We call $\phi_\lambda^\pi$ 
the  
elementary $\pi$-spherical function. 
Since $\boldsymbol{D}^\pi$ contains an elliptic operator, $\phi_\lambda^\pi$ is 
real analytic. Moreover, 
it has an integral representation 
\begin{equation}
\phi^\pi_{\lambda}(g)  =
\int_Ke^{(\lambda-\rho)(H(gk))}\pi(k\kappa(gk)^{-1})dk.  \label{eq:eisenstein}
\end{equation}
Here given $x\in G$, define $\kappa(x)\in K$ and $H(x)\in\mathfrak{a}$ by $x\in \kappa(x)e^{H(x)}N$. 
Notice that $\phi_\lambda^\pi$ is independent of the choice of 
$\varSigma^+$, though the right hand side of \eqref{eq:eisenstein} depends on 
$\varSigma^+$ at first glance. 
Moreover, $\phi_\lambda^\pi$ depends 
holomorphically on $\lambda\in\mathfrak{a}_\mathbb{C}^*$. 

Formula \eqref{eq:eisenstein} is a special case of the integral representations of elementary spherical functions (or more generally the 
{Eisenstein integral}s) given by Harish-Chandra (cf. \cite[\S 6.2.2, \S 9.1.5]{War}, \cite[(42)]{Camporesi2}, \cite[(14.20)]{Kn0}).

\subsection{Harish-Chandra series}

In this subsection, we review the Harish-Chandra expansion of the 
elementary spherical function according to \cite[\S~9.1]{War}. 
We assume $(\pi,V)$ is a small $K$-type. 

Let $C^\infty(G,\pi,\pi)$ denote the space of the 
smooth $\pi$-spherical functions. 
If $f\in C^\infty(G,\pi,\pi)$ then $f|_A$ takes values in $\text{End}_M V\simeq \mathbb{C}$.
Hence we regard 
$\varUpsilon^\pi(f):=f|_A \circ\exp$ as a scalar valued function on $\mathfrak{a}$. 
Let $C^\infty(\mathfrak{a})^W$ denote the space of the $W$-invariant smooth functions on $\mathfrak{a}$.
The restriction map 
$\varUpsilon^\pi$ gives an isomorphism $C^\infty(G,\pi,\pi)\simarrow
C^\infty(\mathfrak{a})^W$ (\cite[Theorem~1.5]{OS}). 

Let $\mathscr R$ be the unital algebra of functions on $\mathfrak a_\reg$
generated by $(1\pm e^{\alpha})^{-1}$ ($\alpha\in\varSigma^+$).
For any $D \in U(\mathfrak g_\bbC)^K$
there exists a unique $W$-invariant 
differential operator $\varDelta^\pi(D) \in \mathscr R \otimes S(\mathfrak a_\bbC)$
such that for any $f \in C^\infty(G,\pi,\pi)$
\[
\varUpsilon^\pi(Df)
=\varDelta^\pi(D)\varUpsilon^\pi(f)
\]
on $\mathfrak a_\reg$ (\cite[Proposition 3.10]{OS}). 
We call $\varDelta^\pi(D)$ the $\pi$-radial part of $D$. The function 
$\Phi=\varUpsilon^\pi(\phi^\pi_\lambda)$ satisfies differential equations
\begin{equation}\label{eqn:derad}
\varDelta^\pi(D)\Phi=\gamma^\pi(D)(\lambda)\Phi\quad 
(D\in U(\mathfrak{g}_\mathbb{C})^K).
\end{equation}

Let $\mathbb{N}\varSigma^+$ denote the set of $\mu\in\mathfrak{a}^*$ of the form 
$\mu=n_1\alpha_1+\cdots +n_r\alpha_r\,\,(n_i\in\mathbb{N})$. 
For $\mu\in\mathbb{N}\varSigma^+\setminus\{0\}$, let 
$\sigma_\mu$ denote the hyperplane
\[
\sigma_\mu=\{\lambda\in\mathfrak{a}_\mathbb{C}^*\,|\,
\langle 2\lambda-\mu,\mu\rangle=0\}.
\]
If $\lambda\not\in\sigma_\mu$ for any $\mu\in\mathbb{N}\varSigma^+\setminus\{0\}$, 
then there exists a unique convergent series solution
\begin{equation}\label{eqn:hcs1}
\Phi_\lambda(H)=e^{(\lambda-\rho)(H)}\sum_{\mu\in\mathbb{N}\varSigma^+}
\Gamma_\mu(\lambda)e^{-\mu(H)}
\quad (H\in \mathfrak{a}_+)
\end{equation}
of \eqref{eqn:derad} 
with $\Gamma_\mu(\lambda)\in\mathbb{C}$ and $\Gamma_0(\lambda)=1$. 
This is a special case of \cite[Theorem~9.1.4.1]{War}. 

By using differential equations \eqref{eqn:derad}, apparent singularities of $\Phi_\lambda(H)$ 
as a function of $\lambda$ is removable unless $\mu=n\alpha$ for some $n\in \mathbb{Z}_{>0}$ and 
$\alpha\in\varSigma^+$ (cf. \cite[Corollary~6.3]{A}, see also 
\cite[Lemma~6.5]{Op:lecture} and \cite[Proposition~7.5]{Hec}). 
For $\mu=n\alpha$, $\lambda\not\in \sigma_{\mu}$ if and only if $\langle\lambda,\alpha^\vee\rangle
\not=n$. Here $\alpha^\vee =2\alpha/\langle\alpha,\alpha\rangle$. 
Thus $\Phi_\lambda$ is defined if $\langle\lambda,\alpha^\vee\rangle\not\in \mathbb{Z}_{>0}$ for 
all $\alpha\in\varSigma^+$. 

If 
$\langle\lambda,\alpha^\vee\rangle\not\in \mathbb{Z}$ for all $\alpha\in\varSigma^+$, then 
$\{\Phi_{w\lambda}\,|\,w\in W\}$ forms a basis of the solution space 
of \eqref{eqn:derad} on $\mathfrak{a}_+$. Thus  
$\varUpsilon^\pi(\phi_\lambda^\pi)$ is a linear combination of 
$\Phi_{w\lambda}\,\,(w\in W)$. 
Since $\phi_{w\lambda}^\pi=\phi_\lambda^\pi$, there exists a constant 
$c^\pi(\lambda)$ such that
\begin{equation}\label{eqn:hcs2}
\varUpsilon^\pi(\phi_\lambda^\pi)(H)
=\sum_{w\in W}c^\pi(w\lambda)\Phi_{w\lambda}(H)
\quad (H\in \mathfrak{a}_+).
\end{equation}

\subsection{Harish-Chandra ${c}$-function}

In this subsection, 
we review 
the Harish-Chandra $c$-function. We refer to \cite{Schiffman}, 
 \cite[\S 9.1.6]{War}, \cite[Chapter 8]{Wallach73}, 
 and \cite[\S 5]{Sekiguchi} for details. 

Let $H\in\mathfrak{a}_+$ and  
$\lambda\in\mathfrak{a}_\mathbb{C}^*$ satisfying $\text{Re}\,\lambda\in \mathfrak{a}_+^*$. 
The leading 
coefficient $c^\pi(\lambda)$ 
of $\varUpsilon^\pi(\phi^\pi_\lambda)$ at infinity in $A_+=\exp\mathfrak{a}_+$ is 
given by the Harish-Chandra ${c}$-function 
(cf. \cite[Theorem 9.1.6.1]{War}, \cite[Theorem~14.7, (14.29)]{Kn0}):
\begin{align}
& \lim_{t\to\infty}e^{t(-\lambda+\rho)(H)}\varUpsilon^\pi(\phi^\pi_\lambda)(e^{tH})={c}^\pi(\lambda), 
\label{eq:limcf1} \\
& {c}^\pi(\lambda)=\int_{\bar{N}}e^{-(\lambda+\rho)(H(\bar{n}))}\pi(\kappa(\bar{n}))d\bar{n}, 
\label{eq:cfgroup}
\end{align}
where the Haar measure on $\bar{N}$ is normalized so that 
\[
\int_{\bar{N}}e^{-2\rho(H(\bar{n}))}d\bar{n}=1.
\]
The integral (\ref{eq:cfgroup}) 
converges absolutely for $\text{Re}\,\lambda\in \mathfrak{a}_+^*$ and extends to a meromorphic function 
on $\mathfrak{a}_\mathbb{C}^*$. 
Notice that ${c}^\pi(\lambda)\in \text{End}_M V\simeq \mathbb{C}$. 

Define 
\[
\varSigma_0=\{\alpha\in\varSigma\,|\,\tfrac12\alpha\not\in\varSigma\}
\]
and $\varSigma_0^+=\varSigma_0\cap \varSigma^+$. 
For $\alpha\in\varSigma_0^+$ let $\mathfrak{g}_{(\alpha)}$ denote the Lie subalgebra of $\mathfrak{g}$ 
generated by $\mathfrak{g}_\alpha$ and $\mathfrak{g}_{-\alpha}$. 
Put $\mathfrak{k}_\alpha=\mathfrak{k}\cap \mathfrak{g}_{(\alpha)}$, $\mathfrak{p}_\alpha=\mathfrak{p}\cap \mathfrak{g}_{(\alpha)}$, 
 $\mathfrak{a}_\alpha=\mathfrak{a}\cap \mathfrak{g}_{(\alpha)}$, $\mathfrak{n}_\alpha=\mathfrak{g}_\alpha+\mathfrak{g}_{2\alpha}$, and 
$\bar{\mathfrak{n}}_{\alpha}=\theta\mathfrak{n}_\alpha$. 
For $\alpha\in \varSigma_0^+$, let $G_\alpha,\,K_\alpha,\,A_\alpha,\,N_\alpha$, and $\bar{N}_{\alpha}$ denote the analytic subgroups of $G$ 
corresponding to $\mathfrak{g}_{(\alpha)},\,\mathfrak{k}_\alpha,\,
\mathfrak{a}_\alpha,\,\mathfrak{n}_\alpha$, and $\bar{\mathfrak{n}}_{\alpha}$, respectively. 
We have the Iwasawa decomposition $G_\alpha=K_\alpha A_\alpha N_\alpha$. Put 
$\rho_\alpha=\frac12(\boldsymbol{m}_\alpha+2\boldsymbol{m}_{2\alpha})\alpha$. 
Let $d\bar{n}_\alpha$ denote the Haar measure on $\bar{N}_\alpha$ 
normalized so that 
\[
\int_{\bar{N}_\alpha}e^{-2\rho_\alpha(H(\bar{n}_\alpha))}d\bar{n}_\alpha=1. 
\]

Let $w^*\in W$ be the element such that $w^*(\varSigma^+)=-\varSigma^+$. 
Let $w^*=s_{i_m}\cdots s_{i_2}s_{i_1}$ be a reduced expression, where $m$ denotes 
the length of $w^*$ and $1\leq i_k\leq r\,\,(1\leq k\leq m)$. Put 
$\beta_k=s_{i_1}\cdots s_{i_{k-1}}\alpha_{i_k}$. Then we have 
$\varSigma_0^+=\{\beta_1,\beta_2,\dots,\beta_m\}$ (cf. \cite[Ch. IV Corollary 6.11]{Hel2}). 
We have the decomposition $\bar{N}=\bar{N}_{\beta_1}\cdots \bar{N}_{\beta_m}$,  
the product map being a diffeomorphism. 
Moreover, there exists a positive constant $c_0$ such that 
\begin{equation}\label{eqn:normalize}
d\bar{n}=c_0\, d\bar{n}_{\beta_1}\cdots d\bar{n}_{\beta_m}.
\end{equation}

For $\alpha\in\varSigma_0^+$ define
\[
c_\alpha^\pi(\lambda)=\int_{\bar{N}_\alpha}e^{-(\lambda+\rho_\alpha)(H(\bar{n}_\alpha))}
\pi(\kappa(\bar{n}_\alpha))d\bar{n}_\alpha.
\]

We have the following product formula (\cite[Theorem 1.2]{Schiffman}, 
\cite[\S  8.11.6]{Wallach73}, \cite[\S 9.1.6]{War}, \cite[Theorem 5.1]{Sekiguchi}).

\begin{thm}\label{thm:prod}
$c^\pi(\lambda)=c_0 \,c_{\beta_1}^\pi\!(\lambda)\cdots c_{\beta_m}^\pi\!(\lambda)$.
\end{thm}

For the case of the trivial $K$-type $\pi=\text{triv}$, 
$c_{\beta_i}^\text{triv}$ can be written
 explicitly by the classical Gamma function and we have 
 the Gindikin and Karpelevi\v{c} product formula for ${c}^\text{triv}(\lambda)$ 
 (cf. \cite{GK}, see also \cite[Ch IV, \S 6]{Hel2} and \cite[\S 9.1.7]{War}). 
 Note that the constant $c_0$ in 
 \eqref{eqn:normalize} is determined explicitly by the 
 Gindikin and Karpelevi\v{c} formula for $c^{\text{triv}}(\lambda)$. 
 
The ${c}$-function for a one-dimensional $K$-type $\pi$ 
of a group $G$ of Hermitian type is also given 
explicitly by the Gamma function (\cite{Sc}, \cite{S:eigen}). 

In \cite{OS} we give an explicit formula  
of $c^\pi(\lambda)$ for each 
simple $G$ and each small $K$-type $\pi$, with one exception for $G$ of type $G_2$ 
and a certain small $K$-type $\pi$. 
The method we use is to relate the $\pi$-elementary spherical 
function $\phi_\lambda^\pi$ with 
the Heckman-Opdam hypergeometric function, 
instead of computing the integral (\ref{eq:cfgroup}) by using Theorem~\ref{thm:prod}. 
Heckman~\cite[Chapter 5]{Hec} gives in this way an explicit formula of ${c}^\pi(\lambda)$ 
for a one-dimensional $K$-type $\pi$ when the group $G$ is  of Hermitian type. 

In \S~\ref{subsec:g2} we give an explicit formula of $c^\pi(\lambda)$ 
for $G$ of type $G_2$ and each small $K$-type $\pi$ by using Theorem~\ref{thm:prod}. 

\subsection{$\pi$-spherical transform}

Let $dH$ denote the Euclidean measure on $\mathfrak{a}$ with respect to the Killing form. 
Define $\delta_{G/K}=\prod_{\alpha \in \varSigma^+}
|2\sinh \alpha|^{\bsm_\alpha}$. 
We normalize the Haar measure $dg$ on $G$ so that
\begin{equation*}
\int_G \psi(g)dg=\frac1{\#W}\int_{\mathfrak a} \psi(e^H)\, \delta_{G/K}(H)dH
\end{equation*}
for any compactly supported continuous $K$-bi-invariant function $\psi$ on $G$
(cf.~\cite[Ch.~I, Theorem 5.8]{Hel2}). 

Let $C^\infty_c(G,\pi,\pi)$ be the subspace of $C^\infty(G,\pi,\pi)$
consisting of the compactly supported smooth $\pi$-spherical functions. 
The $\pi$-spherical transform of $f\in C_c^\infty(G,\pi,\pi)$ is the $\text{End}_M V\simeq \mathbb{C}$-valued 
function ${f}^\wedge$ on $\mathfrak{a}_\mathbb{C}^*$ defined by
\begin{equation}
{f}^\wedge(\lambda)=\int_G \phi_\lambda^\pi(g^{-1})f(g)dg.
\end{equation}
The $\pi$-spherical transform $f\mapsto f^\wedge(\lambda)$ is a homomorphism from the commutative convolution 
algebra $C_c^\infty(G,\pi,\pi)$ to $\mathbb{C}$ (cf. \cite{Camporesi2}). 
It 
 is a special case of the Fourier transform given by  
Arthur~\cite{A} (see also \cite[\S 3]{BS3}). %
By the identification $C_c^\infty(G,\pi,\pi)\simeq C_c^\infty(\mathfrak{a})^W$, 
the 
$\pi$-spherical transform $f^\wedge$ of $f\in C_c^\infty(\mathfrak{a})^W$ is given by  
\begin{equation}
{f}^\wedge(\lambda)=\frac{1}{\# W}\int_{\mathfrak{a}}f(H)
\varUpsilon^\pi(\phi_{-\lambda}^\pi)(H)\delta_{G/K}(H)\,dH. 
\end{equation}
We normalize the Haar measure $d\lambda$ on $\sqrt{-1}\mathfrak{a}^*$
 so that the Euclidean Fourier transform and its inversion are given by 
\[
\tilde{f}(\lambda)=\int_{\mathfrak{a}}f(H)e^{-\lambda(H)}dH, \qquad
f(H)=\int_{\sqrt{-1}\mathfrak{a}^*} \tilde{f}(\lambda)e^{\lambda(H)}d\lambda. 
\]

Let $\eta_1$ be a point in $-\overline{\mathfrak{a}_+^*}$ such 
that $c^\pi(-\lambda)^{-1}$ is a regular 
function of $\lambda$ for $\text{Re}\,\lambda\in \eta_1-\overline{\mathfrak{a}_+^*}$. 
The existence of such $\eta_1$ follows from an 
explicit formula of the $c$-function for each small $K$-type, 
which is determined by \cite{OS} and \S\ref{sec:g2} for $G_2$.
It also follows from a general result on the Harish-Chandra $c$-function 
due to Cohn \cite{Cohn}. 

Let $F=f^\wedge$. 
Following \cite[Chapter II, \S 1]{A}, 
define the function $F^\vee(H)$ on $\mathfrak{a}_+$ by
\begin{equation}\label{eqn:invft}
F^\vee(H)=
\int_{\eta_1+\sqrt{-1}\mathfrak{a}^*} F(\lambda)\Phi
_\lambda(H)c^\pi(-\lambda)^{-1}d\lambda.
\end{equation}
The integral \eqref{eqn:invft} converges and is independent of $\eta_1$ 
(cf. \cite[Chapter II, \S 1]{A}). 
$F^\vee$ defined above coincides with that given by Arthur, 
because  $\phi_\lambda^\pi$ is $W$-invariant in $\lambda$ and the 
Harish-Chandra $\mu$-function associate with a minimal parabolic subgroup 
in our case is given by 
$c^\pi(\lambda)^{-1}c^\pi(-\lambda)^{-1}$ (cf. \cite[\S~10.5]{Wal}). 
The following theorem is a special case of \cite[Chapter III, Theorem 3.2]{A}.
It is also a special case of \cite[Theorem 1.1]{BS2}, since $G$ is a 
semisimple 
symmetric space for $G\times G$ (cf. \cite{BS3}).


\begin{thm}\label{thm:inv}
For $f\in C_c^\infty(\mathfrak{a})^W$ we have
\[
f(H)=({f}^\wedge)^\vee(H)\quad (H\in\mathfrak{a}_+).
\]
\end{thm}

If $c^\pi(-\lambda)^{-1}$ is a regular function of $\lambda$ for 
$\text{Re}\,\lambda\in -\overline{\mathfrak{a}_+^*}$, 
then we can take $\eta_1=0$ and by \eqref{eqn:hcs2} and $W$-invariance of 
$c^\pi(\lambda)^{-1}c^\pi(-\lambda)^{-1}$ in 
$\lambda\in\sqrt{-1}\mathfrak{a}^*$, we have 
\begin{equation}\label{eqn:invcont}
f(H)
=
\frac{1}{\# W}\int_{\sqrt{-1}\mathfrak{a}^*}
f^\wedge (\lambda)\varUpsilon^\pi(\phi^\pi_\lambda)(H)
|c^\pi(\lambda)|^{-2}d\lambda\quad (H\in\mathfrak{a}^*). 
\end{equation}

In \cite[Corollary 7.6]{OS} 
we prove the formula \eqref{eqn:invcont} by using the inversion formula 
of the hypergeometric Fourier transform due to Opdam~\cite{Op:Cherednik},
under the assumption that $\pi$ is a small $K$-type of
a real simple $G$ which is not in the following list:

\smallskip
\noindent
(1) $\mathfrak{g}=\mathfrak{sp}(p,1),\,\pi=\pi_n\circ \text{pr}_2$ 
($\pi_n$ : $n$-dimensional irred. rep. of $\text{Sp}(1))$,  
\\
(2) $\mathfrak{g}=\mathfrak{so}(2r,1)$,\\ 
\phantom{(3)} $\pi=\pi_s^{\pm}$ : 
irred. rep. of $\text{Spin}(2r)$ with h.w. $(s/2,\cdots s/2,\pm s/2)\,\,
(s\in \mathbb{N})$,\\
(3) $\mathfrak{g}=\mathfrak{so}(p,q)\quad (p>q\geq 3,\,\,\text{$p$ : even, 
$q$ : odd)}$,\\
\phantom{(3)} $\pi=\sigma\circ\text{pr}_1\,\,(\sigma$ : one of half spin representation of $\text{Spin}(p)$),\\
(4) $\mathfrak{g}$\,:\,Hermitian type, $\pi$ : one dimensional $K$-type, 
\\
(5) $\mathfrak{g}=\mathfrak{g}_2$, $\pi=\pi_2$ (see \S \ref{sec:g2} for the definition).

\medskip
Though the case (3) is not covered by \cite[Corollary 7.6]{OS}, 
the formula \eqref{eqn:invcont} holds in this case, since 
$c^\pi(-\lambda)^{-1}$ is a regular function of $\lambda$ for 
$\text{Re}\,\lambda\in -\overline{\mathfrak{a}_+^*}$ as we mention in 
the final part of \cite{OS}. 

If the parameter of the small $K$-type is 
``large enough'' in the cases (1), (2), and (4), 
then $c^\pi(-\lambda)^{-1}$ has singularities in 
$\text{Re}\,\lambda\in -\overline{\mathfrak{a}_+^*}$ and 
we must take account of residues during the contour 
shift $\eta_1+\sqrt{-1}\mathfrak{a}^*\to \sqrt{-1}\mathfrak{a}^*$. 
The most continuous part of the spectrum is given by the right hand side 
of \eqref{eqn:invcont}. In addition, there are spectra with low dimensional 
supports. The residue calculus in the case (4) is done by \cite{SPlancherel}. 
For the cases (1) and (2), $\dim\mathfrak{a}=1$ and the residue calculus is 
easy to proceed. Also these cases are covered by the inversion formula 
of the Jacobi transform (cf. \cite[Appendix 1]{FJ}, \cite{Koornwinder}). 
See also \cite{vDP} and \cite{Shyperbolic} for the case (1). 

We will discuss the case (5) in the next section. 

\section{The case of $G_2$}
\label{sec:g2}

\subsection{Notation and preliminary results}\label{subsec:g2prem}
Let $\mathfrak{g}$ be the simple split real Lie algebra of type $G_2$ and $G$ the connected simply connected  
Lie group with the Lie algebra $\mathfrak{g}$. 
$G$ is the double cover of the adjoint group of $\mathfrak{g}$. 
Let $K$ be a maximal compact subgroup of $G$ and 
$\mathfrak{k}$ the Lie algebra of $K$. Then 
$K\simeq \SU(2)\times \SU(2)$ and $\mathfrak{k}\simeq \mathfrak{su}(2)\oplus \mathfrak{su}(2)$. 

Let $\mathfrak{t}$ be a maximal abelian subalgebra of $\mathfrak{k}$. 
Then $\mathfrak{t}$ is a Cartan subalgebra 
of $\mathfrak{g}$. Let $\Delta$ and $\Delta_K$ denote the root system for $(\mathfrak{g}_\mathbb{C},\mathfrak{t}_\mathbb{C})$ 
and $(\mathfrak{k}_\mathbb{C},\mathfrak{t}_\mathbb{C})$, respectively. 
Then $\Delta$ is a root system of type $G_2$. 
We choose a positive system $\Delta^+\subset \Delta$
so that its base contains a short compact root $\beta_1$.
The other simple root, say $\beta_2$, is a long noncompact root.
If we put $\Delta_K^+=\Delta_K\cap \Delta^+$ then
\begin{align*}
& \Delta^+=\{\beta_1,\beta_2,\beta_1+\beta_2,2\beta_1+\beta_2,3\beta_1+\beta_2,3\beta_1+2\beta_2\},\\
& \Delta_K^+=\{\beta_1,3\beta_1+2\beta_2\}.
\end{align*}
We fix an inner product on $\sqrt{-1}\mathfrak{t}^*$ such that 
$(\beta_1,\beta_1)=2$. Then $(\beta_2,\beta_2)=6$ and $(\beta_1,\beta_2)=-3$.
We let $K=K_1\times K_2$ with $K_i\simeq \SU(2)\,\,(i=1,2)$
assuming that
$\beta_1$ (resp. $3\beta_1+2\beta_2$) is a 
root for $((\mathfrak{k}_1)_\mathbb{C}, 
(\mathfrak{t}\cap\mathfrak{k}_1)_\mathbb{C})$ (resp. 
$((\mathfrak{k}_2)_\mathbb{C}, 
(\mathfrak{t}\cap\mathfrak{k}_2)_\mathbb{C})$).
Let  
$\pr_i$ denote the projection of $K$ to $K_i\,\,(i=1,2)$. 

The classification of the small $K$-types for $G$ is given as follows:
\begin{thm}[{\cite[Theorem 1]{SWL}}]\label{thm:G2}
The non-trivial small $K$-types are $\pi_1:=\sigma\circ\pr_1$ and  $\pi_2:=\sigma\circ\pr_2$. 
Here $\sigma$ is the two-dimensional irreducible representation of $\SU(2)$. 
\end{thm}

A discrete series representation of $G$ is an irreducible representation of $G$ realized as a closed 
$G$-invariant subspace of the 
left regular representation on $L^2(G)$. 

\begin{lem}
Let $G$ be as above.
Then no small $K$-type appears in any discrete series representation of $G$.
\end{lem}
\begin{proof}

If $\pi$ is the trivial $K$-type or $\pi=\pi_1$, then 
it follows from the Plancherel formula for the 
$\pi$-spherical functions (cf. \cite{Hel2}, \cite{OS}) 
that there are no discrete series representations having 
the $K$-type $\pi$. 
 
Next let us discuss the case of $\pi_2$.
The positive open chamber $(\sqrt{-1}\mathfrak{t}^*)^+$ for $\Delta_K^+$
contains the following three open chambers for $\Delta$:
\begin{align*}
(\sqrt{-1}\mathfrak{t}^*)^+_1&:=\{\lambda\in (\sqrt{-1}\mathfrak{t}^*)^+\,|\, (\lambda,\beta_2)>0\},\\
(\sqrt{-1}\mathfrak{t}^*)^+_2&:=\{\lambda\in (\sqrt{-1}\mathfrak{t}^*)^+\,|\, (\lambda,\beta_2)<0\text{ and }(\lambda,\beta_1+\beta_2)>0\},\\
(\sqrt{-1}\mathfrak{t}^*)^+_3&:=\{\lambda\in (\sqrt{-1}\mathfrak{t}^*)^+\,|\, (\lambda,\beta_1+\beta_2)<0\}.
\end{align*}
Let $\Delta_i^+$ be the corresponding positive systems ($i=1,2,3$).
Note that $\Delta_1^+=\Delta^+$. 
If we put $\delta_i=\frac12\sum_{\beta\in\Delta^+_i}\beta$ then
\[
\delta_1=5\beta_1+3\beta_2,\quad
\delta_2=5\beta_1+2\beta_2,\quad
\delta_3=4\beta_1+\beta_2.
\]
On the other hand,
$\delta_K:=\frac12\sum_{\beta\in\Delta_K^+}\beta=2\beta_1+\beta_2.$
Now suppose $\pi_2$ appears in a discrete series representation with Harish-Chandra parameter 
$\lambda\in \sqrt{-1}\mathfrak{t}^*$.
We may assume $\lambda\in (\sqrt{-1}\mathfrak{t}^*)^+_i$
for $i=1,2$, or $3$.
Since the highest weight of $\pi_2$ is $\frac32\beta_1+\beta_2$,
it follows from \cite[Theorem~8.5]{AS} that 
\[
\frac32\beta_1+\beta_2=\lambda+\delta_i-2\delta_K+\sum_{\beta\in\Delta^+_i}
n_\beta \beta\quad \text{for some }n_\beta\in\mathbb{N}. 
\]
If $i=1$ then this reduces to
\[
\lambda=\left(\frac12-c_1\right)\beta_1-c_2\beta_2\quad\text{for some }c_1,\,c_2\in\mathbb{N}.
\]
Since $(\lambda,\beta_j)>0\,\,(j=1,2)$, we have $1-2c_1+3c_2>0$ and $-\frac32+3c_1-6c_2>0$, 
which are impossible.
If $i=2$ or $3$ then we can also deduce a contradiction in a similar way.
\end{proof}

\subsection{Harish-Chandra ${c}$-function for $G_2$
}
\label{subsec:g2}

The restricted root system $\varSigma=\varSigma(\mathfrak{g},\mathfrak{a})$ is 
a root system of type $G_2$. 
For all $\alpha\in\varSigma$ we have 
$\mathfrak{g}_{(\alpha)}\simeq \mathfrak{sl}(2,\mathbb{R})$, since 
$\boldsymbol{m}_\alpha=\dim\mathfrak{g}_\alpha=1$ and $2\alpha\not\in\varSigma$. 

We recall the ${c}$-function for $\mathfrak{g}=\mathfrak{sl}(2,\mathbb{R})$. 
Put 
\[
h=\begin{pmatrix} 1 & 0 \\ 0 & -1\end{pmatrix},\quad 
e=\begin{pmatrix} 0 & 1 \\ 0 & 0\end{pmatrix},\quad 
f=\begin{pmatrix} 0 & 0 \\ 1 & 0\end{pmatrix}. 
\]
Then $\{h,e,f\}$ is a basis of $\mathfrak{sl}(2,\mathbb{R})$ and also 
forms an $\mathfrak{sl}_2$-triple. 
We put $\mathfrak{a}=\mathbb{R}h$ and $\mu=\lambda(h)$ for $\lambda\in\mathfrak{a}_\mathbb{C}^*$. 
By \cite[Remark 7.3]{Sc}, 
the ${c}$-function  for $\mathfrak{sl}(2,\mathbb{R})$ with a one-dimensional 
$\mathfrak{k}$-type $\pi$ of the weight $\nu\in\mathbb{Q}$ for $\sqrt{-1}(e-f)$ is given by 
\begin{equation}\label{eq:cfsl2}
{c}^\pi(\lambda)=\frac{2^{1-\mu}\varGamma(\mu)}{\varGamma(\frac12(\mu+1+\nu))
\varGamma(\frac12(\mu+1-\nu))}.
\end{equation}

Now we come back to the case of $G_2$. For $\alpha\in\varSigma^+$ 
choose $e_\alpha\in \mathfrak{g}_{\alpha}$ so 
that $\{\alpha^\vee,\,e_\alpha,\,-\theta e_\alpha\}$ is an $\mathfrak{sl}_2$-triple. 
Put $t_\alpha:=e_\alpha+\theta e_\alpha\in \sqrt{-1}\mathfrak{k}_\alpha$.  
If $\alpha\in \varSigma^+$ is a long root, then the 
possible weights of $t_\alpha$ for 
$\pi_i\,(i=1,2)$ are $\pm \frac12$ by \cite[Lemma 4.2]{SWL}. 
If $\alpha\in \varSigma^+$ is a short root, then the possible weights of $t_\alpha$  
for $\pi_1$ (resp. $\pi_2$) are $\pm\frac12$ 
(resp. $\pm\frac32$) by \cite[Lemma 4.3]{SWL}. 
Since \eqref{eq:cfsl2} remains unchanged if we replace $\nu$ by $-\nu$, 
$c_\alpha^{\pi_i}(\lambda)$ is a scalar operator for each $\alpha\in \varSigma$ and $i=1,\,2$. 

Let $\varSigma_\text{long}^+$ and $\varSigma_\text{short}^+$ denote 
the sets of the long and short positive roots, respectively. 
Define $\lambda_\alpha=\langle\lambda,\alpha^\vee\rangle$ 
for $\lambda\in\mathfrak{a}_\mathbb{C}^*$ and 
$\alpha\in\varSigma^+$. 
It follows from Theorem~\ref{thm:prod}, 
\eqref{eq:cfsl2}, and the proof of \cite[Ch. IV, Theorem~6.13]{Hel2} 
that
\begin{align*}
{c}^{\text{triv}}(\lambda) & 
=
c_0\!
\prod_{\alpha\in\varSigma^+}\frac{2^{1-\lambda_\alpha}
\varGamma(\lambda_\alpha)}
{\varGamma(\frac12\lambda_\alpha+\frac12)^2}
=
c_0\!
\prod_{\alpha\in\varSigma^+}\frac{\pi^{-\frac12}
\varGamma(\frac12\lambda_\alpha)}
{\varGamma(\frac12\lambda_\alpha+\frac12)}, 
\\
{c}^{\pi_1}(\lambda) & =c_0\!\prod_{\alpha\in\varSigma^+} 
\frac{2^{1-\lambda_\alpha}\varGamma(\lambda_\alpha)}{\varGamma(\frac12(\lambda_\alpha+\frac32))
\varGamma(\frac12(\lambda_\alpha+\frac12))} 
  =c_0\!\prod_{\alpha\in\varSigma^+} 
\frac{2^{\frac12}\pi^{-\frac12}\varGamma(\lambda_\alpha)}
{\varGamma(\lambda_\alpha+\frac12)}, 
\\
{c}^{\pi_2}(\lambda) & 
=c_0\!\!
\prod_{\alpha\in\varSigma_\text{long}^+} \!\!
\frac{2^{1-\lambda_\alpha}\varGamma(\lambda_\alpha)}{\varGamma(\frac12(\lambda_\alpha+\frac32))
\varGamma(\frac12(\lambda_\alpha+\frac12))} 
\prod_{\beta\in\varSigma_\text{short}^+} \!\!
\frac{2^{1-\lambda_\beta}\varGamma(\lambda_\beta)}{\varGamma(\frac12(\lambda_\beta+\frac52))
\varGamma(\frac12(\lambda_\beta-\frac12))} \\
&=c_0\!\!
\prod_{\alpha\in\varSigma_\text{long}^+} \!\!
\frac{2^{\frac12}\pi^{-\frac12}\varGamma(\lambda_\alpha)}
{\varGamma(\lambda_\alpha+\frac12)}
\prod_{\beta\in\varSigma_\text{short}^+} \!\!
 \frac{2^{\frac12}\pi^{-\frac12}\left(\lambda_\beta-\frac12\right)\varGamma(\lambda_\beta)}
{\varGamma(\lambda_\beta+\frac32)} .
\end{align*}

The value of the constant $c_0$ is determined by $c^{\text{triv}}(\rho)=1$. 
We have $c_0=2\pi^2$ by direct computation. 
Thus we have the following 
theorem. 

\begin{thm}\label{thm:cfg2}
\begin{align*}
{c}^{\text{\rm triv}}(\lambda) & 
=
\frac{2}{\pi}
\prod_{\alpha\in\varSigma^+}\frac{
\varGamma(\frac12\lambda_\alpha)}
{\varGamma(\frac12\lambda_\alpha+\frac12)}, 
\\
{c}^{\pi_1}(\lambda) & 
  =\frac{16}{\pi}\prod_{\alpha\in\varSigma^+} 
\frac{\varGamma(\lambda_\alpha)}
{\varGamma(\lambda_\alpha+\frac12)}, 
\\
{c}^{\pi_2}(\lambda) 
&=\frac{16}{\pi}
\prod_{\alpha\in\varSigma_\text{\rm long}^+} 
\frac{\varGamma(\lambda_\alpha)}
{\varGamma(\lambda_\alpha+\frac12)}
\prod_{\beta\in\varSigma_\text{\rm short}^+} 
 \frac{\left(\lambda_\beta-\frac12\right)\varGamma(\lambda_\beta)}
{\varGamma(\lambda_\beta+\frac32)} .
\end{align*}
\end{thm}

The formula for $c^{\text{triv}}(\lambda)$ in Theorem~\ref{thm:cfg2} is a special case of the 
Gindikin-Karpelevi\v{c} formula (cf. \cite{GK}, \cite[Ch.~IV, Theorem~6.13]{Hel2}). 
The formula for $c^{\pi_1}(\lambda)$ 
is given in \cite{OS} by use of a different method. 
The formula for $c^{\pi_2}(\lambda)$ is new. 

\subsection{$\pi$-spherical transform}

An inversion formula for the $\pi$-spherical transform is given by 
Theorem~\ref{thm:inv}. We must shift the contour of integration from 
$\eta_1+\sqrt{-1}\mathfrak{a}^*$ to $\sqrt{-1}\mathfrak{a}^*$ and get a 
globally defined function on $\mathfrak{a}$. 
We refer to \cite[Ch. II, Ch. III]{A} for the general residue scheme 
(see also \cite{Oshima81a,BS,BS2,BS2000,BS3,BS4}). 

For $\pi=\text{triv}$ and 
$\pi_1$, $c^\pi(-\lambda)^{-1}$ is a regular function of $\lambda$ for 
$\text{Re}\,\lambda\in -\overline{\mathfrak{a}_+^*}$, hence
 the inversion formula is given by \eqref{eqn:invcont} for these 
small $K$-types (cf. \cite[Ch~IV, Theorem~7.5]{Hel2}, \cite[Corollary~7.6]{OS}). 

For $\pi=\pi_2$, there appear singularities during the contour shift and 
we must take account of residues. 
The function $c^{\pi_2}(-\lambda)^{-1}$ for $\text{Re}\,\lambda\in
\mathfrak{a}_+^*$ 
has singularities along 
lines $\lambda_\beta=-\frac12\,\,(\beta\in\varSigma^+_{\text{short}})$. 
Figure~1 illustrates singular lines 
$\lambda_{\alpha_1}=-\frac12,\,\lambda_{\alpha_1+\alpha_2}=-\frac12$, and 
$\lambda_{2\alpha_1+\alpha_2}=-\frac12$ 
 (dashed) for $c^{\pi_2}(-\lambda)^{-1}$ 
and $-\overline{\mathfrak{a}_+^*}$ 
(shaded region). 
Here $\alpha_1$ and $\alpha_2$ are the simple roots of $\varSigma^+$ ($||\alpha_1||<||\alpha_2||$).
These singular lines divide $-\overline{\mathfrak{a}_+^*}$ into 
the following four regions (indicated in Figure 1):
\begin{figure}
\begin{center}
\includegraphics[trim=0 105 0 105,width=9.2cm]{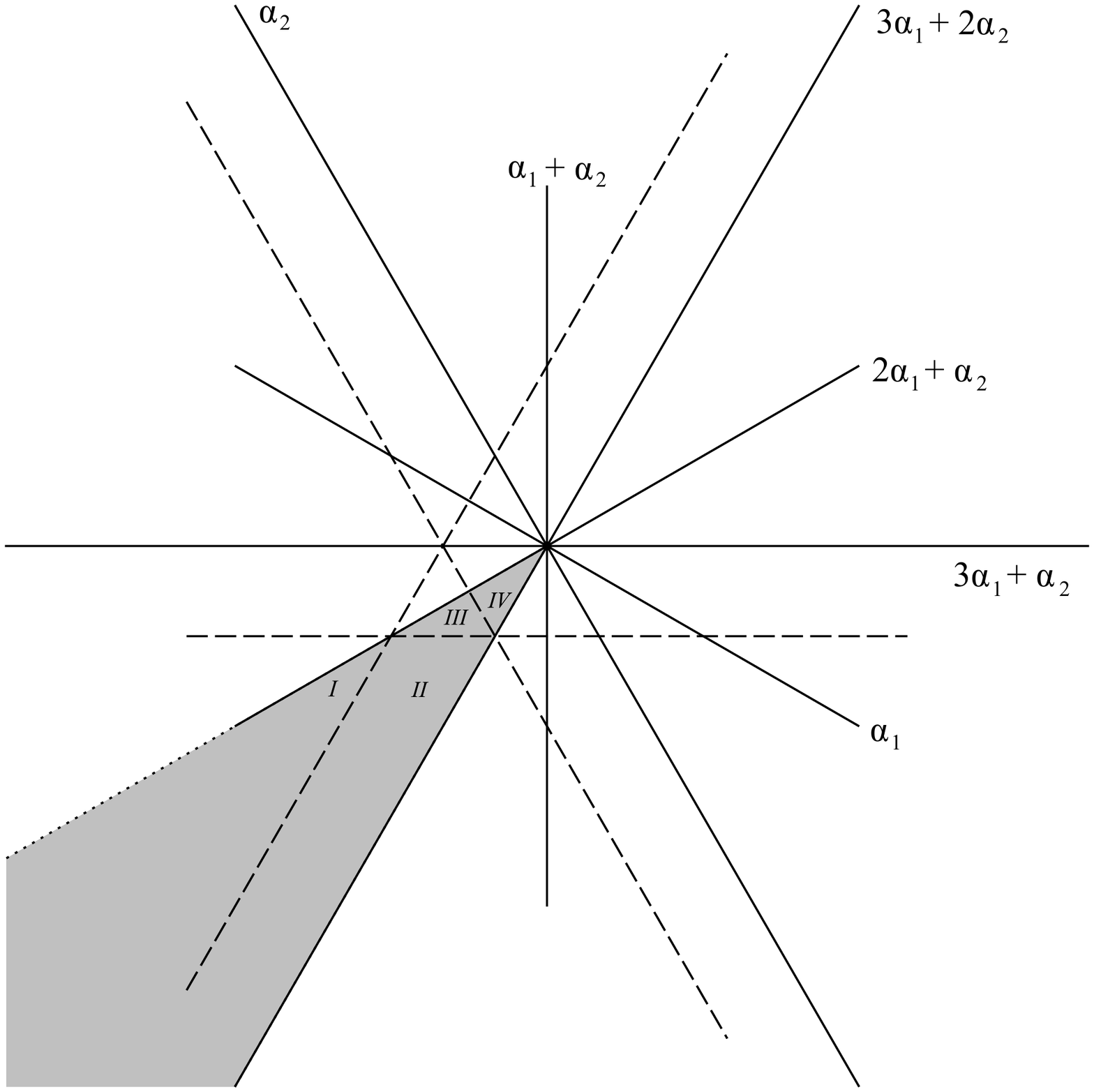}
\caption{singular lines}
\end{center}
\end{figure}
\begin{align*}
& \text{\emph{I}} \,:\, \lambda_{\alpha_1}<-\frac12,\quad \lambda_{\alpha_2}\leq 0, \\
& \text{\emph{II}} \,:\, -\frac12<\lambda_{\alpha_1}\leq 0,\quad \lambda_{\alpha_1+\alpha_2}<-\frac12, \\
& \text{\emph{III}} \,:\, -\frac12<\lambda_{\alpha_1+\alpha_2},\quad 
\lambda_{2\alpha_1+\alpha_2}<-\frac12,\quad \lambda_{\alpha_2}\leq 0, \\
& \text{\emph{IV}} \,:\,\lambda_{2\alpha_1+\alpha_2}>-\frac12,\quad \lambda_{\alpha_1}\leq 0,\quad 
\lambda_{\alpha_2}\leq 0. 
\end{align*}

First $\eta_1\in-\mathfrak{a}_+^*$ in \eqref{eqn:invft} 
lies in the region \emph{I}. 
We choose $\eta_2,\,\eta_3$, and $\eta_4$ 
in the regions \emph{II,\,III}, and \emph{IV}, respectively. We may take $\eta_4=0$. 
We shift the contour of 
integration from $\eta_1+\sqrt{-1}\mathfrak{a}^*$ to $\eta_2+\sqrt{-1}\mathfrak{a}^*$ 
and so on, and finally to $\eta_4+\sqrt{-1}\mathfrak{a}^*=
\sqrt{-1}\mathfrak{a}^*$, picking up residues. 
Define
\[
F^\vee_i(H)
=\int_{\eta_i+\sqrt{-1}\mathfrak{a}^*}f^\wedge (\lambda)\Phi_{\lambda}(H)
c^{\pi_2}(-\lambda)^{-1}d\lambda\quad (1\leq i\leq 4).
\]

We regard 
$(\lambda_{\alpha_1},\lambda_{3\alpha_1+2\alpha_2})\in\mathbb{C}^2$ as 
a coordinate on $\mathfrak{a}_\mathbb{C}^*$. 
Define 
\begin{equation}
c_1=\frac{||\alpha_1||\,||3\alpha_1+2\alpha_2||}{4}.
\end{equation}
For $\eta\in\mathfrak{a}^*$, the normalized measure $d\lambda$ on $\eta+\sqrt{-1}\mathfrak{a}^*$ is given by
\[
d\lambda=\frac{c_1}{(2\pi\sqrt{-1})^2}d\lambda_{\alpha_1}d\lambda_{3\alpha_1+2\alpha_2}.
\]

First we change the contour of integration of $F^\vee(H)=F^\vee_1(H)$ from $\eta_1+\sqrt{-1}\mathfrak{a}^*$ 
to $\eta_2+\sqrt{-1}\mathfrak{a}^*$ with respect to the integration in the variable 
$\lambda_{\alpha_1}$. 
By the explicit formula of $c^{\pi_2}(\lambda)$ in Theorem~\ref{thm:cfg2}, 
$f^\wedge (\lambda)\Phi_{\lambda}(H)
c^{\pi_2}(-\lambda)^{-1}$ has a possible simple pole during the change of 
the contour coming from the factor 
$(-\lambda_{\alpha_1}-\frac12)^{-1}$ of $c^{\pi_2}_{\alpha_1}(-\lambda)^{-1}$. 
Thus the difference $F^\vee_1(H)-F^\vee_2(H)$ is
\[
-\frac{c_1}{2\pi\sqrt{-1}}\int_{\mu+\sqrt{-1}\mathbb{R}}
\Res_{
\lambda_{\alpha_1}=-\frac12}\!\left(
f^\wedge (\lambda)\Phi_{\lambda}(H)
c^{\pi_2}(-\lambda)^{-1}\right)\!
d\lambda_{3\alpha_1+2\alpha_2}
\]
for some $\mu$ with $\mu_{\alpha_1}=-\frac12,\,\mu_{\alpha_2}<0$. 
Next we move $\mu_{3\alpha_1+2\alpha_2}$ to $0$ along the 
line $\lambda_{\alpha_1}=-\frac12$. 
Singularities coming from $(-\lambda_\beta-\frac12)^{-1}
\,\,(\beta=\alpha_1+\alpha_2,\,2\alpha_1+\alpha_2)$ are on the walls and 
they are canceled by $\varGamma(-\lambda_\alpha)^{-1}\,\,
(\alpha=\alpha_2,\,\alpha_1+\alpha_2, \text{respectively})$. Thus 
the integrand 
$\Res_{
\lambda_{\alpha_1}=-\frac12}\!\left(
f^\wedge (\lambda)\Phi_{\lambda}(H)
c^{\pi_2}(-\lambda)^{-1}\right)$ is regular for $\lambda_{3\alpha_1+2\alpha_2}\leq 0$.  
Hence we have 
\[
F^\vee_1(H)-F^\vee_2(H)=
-\frac{c_1}{2\pi\sqrt{-1}}\int_{\sqrt{-1}\mathbb{R}}
\Res_{
\lambda_{\alpha_1}=-\frac12}\!\left(
f^\wedge (\lambda)\Phi_{\lambda}(H)
c^{\pi_2}(-\lambda)^{-1}\right)\!
d\lambda_{3\alpha_1+2\alpha_2}.
\]
Similarly, we have
\begin{align*}
& F^\vee_2(H)-F^\vee_3(H)=
-\frac{c_1}{2\pi\sqrt{-1}}\int_{\sqrt{-1}\mathbb{R}}
\Res_{
\lambda_{\alpha_1+\alpha_2}=-\frac12}\!\left(
f^\wedge (\lambda)\Phi_{\lambda}(H)
c^{\pi_2}(-\lambda)^{-1}\right)\!
d\lambda_{3\alpha_1+\alpha_2}, \\
& F^\vee_3(H)-F^\vee_4(H)=
-\frac{c_1}{2\pi\sqrt{-1}}\int_{\sqrt{-1}\mathbb{R}}
\Res_{
\lambda_{2\alpha_1+\alpha_2}=-\frac12}\!\left(
f^\wedge (\lambda)\Phi_{\lambda}(H)
c^{\pi_2}(-\lambda)^{-1}\right)\!
d\lambda_{\alpha_2}.
\end{align*}
By summing up and changing variables, we have
\begin{align*}
F_1^\vee & (H)-F_4^\vee (H) \\ & =-\frac{c_1}{2\pi\sqrt{-1}}
\!\!
\sum_{w\in\{e,s_1,s_2 s_1\}}\!
\int_{\sqrt{-1}\mathbb{R}}
\Res_{
\lambda_{2\alpha_1+\alpha_2}=-\frac12}\!\left(
f^\wedge (\lambda)\Phi_{w\lambda}(H)
c^{\pi_2}(-w\lambda)^{-1}\right)\!
d\lambda_{\alpha_2}. 
\end{align*}
Let $W^{2\alpha_1+\alpha_2}=\{e,s_1,s_2,s_1s_2,s_2s_1,s_2 s_1 s_2\}$. 
By changing variables, we have
\begin{align}
F_1^\vee & (H)  -F_4^\vee (H) \label{eqn:invtemp}
 \\ & =-\frac{c_1}{4\pi\sqrt{-1}}\!\!\!
\sum_{w\in W^{2\alpha_1+\alpha_2}}
\int_{\sqrt{-1}\mathbb{R}}
\Res_{
\lambda_{2\alpha_1+\alpha_2}=-\frac12}\!\left(
f^\wedge (\lambda)\Phi_{w\lambda}(H)
c^{\pi_2}(-w\lambda)^{-1}\right)\!
d\lambda_{\alpha_2} . \notag
\end{align}

Since $W^{2\alpha_1+\alpha_2}=\{w\in W\,|\,w(2\alpha_1+\alpha_2)\in\varSigma^+\}$, $c^{\pi_2}(w\lambda)|_
{\lambda_{2\alpha_1+\alpha_2}=-\frac12}=0$ for 
any $w\in W\setminus W^{2\alpha_1+\alpha_2}$ by Theorem~\ref{thm:cfg2}. 
Notice that the Harish-Chandra expansion \eqref{eqn:hcs2} is valid for 
$\lambda_{2\alpha_1+\alpha_2}=-\frac12,\,
\lambda_{\alpha_2}\in\sqrt{-1}\mathbb{R}\setminus\{0\}$. 
Hence 
\begin{equation}\label{eqn:hcstemp}
\varUpsilon^{\pi_2}(\phi_\lambda^{\pi_2})|_{\lambda_{2\alpha_1+\alpha_2}=-\frac12}
=\sum_{w\in W^{2\alpha_1+\alpha_2}}c^{\pi_2}(w\lambda)\Phi_{w\lambda}|_{\lambda_{2\alpha_1+\alpha_2}=-\frac12}
\end{equation}
for $\lambda_{\alpha_2}\in\sqrt{-1}\mathbb{R}\setminus\{0\}$. 

We write the $c$-function $c^{\pi_2}(\lambda)$ in Theorem~\ref{thm:cfg2} as 
\[
c^{\pi_2}(\lambda)=\frac{16}{\pi}c_l(\lambda)c_s(\lambda)
\]
with
\[
c_l(\lambda)=
\prod_{\alpha\in\varSigma_\text{\rm long}^+} 
\frac{\varGamma(\lambda_\alpha)}
{\varGamma(\lambda_\alpha+\frac12)}, 
\qquad 
c_s(\lambda)=
\prod_{\beta\in\varSigma_\text{\rm short}^+} 
 \frac{\left(\lambda_\beta-\frac12\right)\varGamma(\lambda_\beta)}
{\varGamma(\lambda_\beta+\frac32)} .
\]  
Notice that the functions  
$(c_l(\lambda)c_l(-\lambda))^{-1}$ and $(c_s(\lambda)c_s(-\lambda))^{-1}$ 
are $W$-invariant.
\begin{lem}
We have
\begin{equation}\label{eqn:restemp0}
(c_l(\lambda)c_l(-\lambda))^{-1}|_
{\lambda_{2\alpha_1+\alpha_2}=-\frac12}
=\frac{(4\lambda_{\alpha_2}^3-\lambda_{\alpha_2})\sin\pi\lambda_{\alpha_2}}
{16\cos\pi\lambda_{\alpha_2}}
\end{equation}
and 
\begin{equation}\label{eqn:restemp}
c_s(w\lambda)^{-1}|_
{\lambda_{2\alpha_1+\alpha_2}=-\frac12}
\Res_{
\lambda_{2\alpha_1+\alpha_2}=-\frac12}(c_s(-w\lambda)^{-1})
=\frac{36\lambda_{\alpha_2}^2-1}{32\pi}
\end{equation}
for any $w\in W^{2\alpha_1+\alpha_2}$.
\end{lem}
\begin{proof}
We show only \eqref{eqn:restemp} because \eqref{eqn:restemp0} can be deduced in a similar way.
Since the left hand side of \eqref{eqn:restemp} is the residue of  
the $(c_s(\lambda)c_s(-\lambda))^{-1}$ 
as a function of $\lambda_{2\alpha_1+\alpha_2}$ 
at $\lambda_{2\alpha_1+\alpha_2}=-\frac12$,
it suffices to show \eqref{eqn:restemp} for $w=1$.
By elementary calculation we have
\[
 \frac{\left(z-\frac12\right)\varGamma(z)}
{\varGamma(z+\frac32)}
\cdot
 \frac{\left(-z-\frac12\right)\varGamma(-z)}
{\varGamma(-z+\frac32)}
=-\frac{\cos\pi z}{z\sin\pi z}\quad(z\in\bbC).
\]
Using
$\lambda_{\alpha_1}=\frac12\lambda_{2\alpha_1+\alpha_2}-\frac32\lambda_{\alpha_2}$
and
$\lambda_{\alpha_1+\alpha_2}=\frac12\lambda_{2\alpha_1+\alpha_2}+\frac32\lambda_{\alpha_2}$
we calculate
\begin{align*}
\Res_{
\lambda_{2\alpha_1+\alpha_2}=-\frac12}&((c_s(\lambda)c_s(-\lambda))^{-1})\\
&=\biggl.\biggl(-\frac{z\sin\pi z}{\cos\pi z}\biggr)\biggr|_{z=-\frac14-\frac32\lambda_{\alpha_2}}
\biggl.\biggl(-\frac{z\sin\pi z}{\cos\pi z}\biggr)\biggr|_{z=-\frac14+\frac32\lambda_{\alpha_2}}
\Res_{z=-\frac12}\biggl(-\frac{z\sin\pi z}{\cos\pi z}\biggr)\\
&=
\frac{1-36\lambda_{\alpha_2}^2}{16}
\cdot\frac{\sin\pi\Bigl(-\frac14-\frac32\lambda_{\alpha_2}\Bigr)
\sin\pi\Bigl(-\frac14+\frac32\lambda_{\alpha_2}\Bigr)}%
{\cos\pi\Bigl(-\frac14-\frac32\lambda_{\alpha_2}\Bigr)
\cos\pi\Bigl(-\frac14+\frac32\lambda_{\alpha_2}\Bigr)}
\cdot\biggl(-\frac1{2\pi}\biggr).
\end{align*}
In the final expression the second factor reduces to $1$.
\end{proof}

Thus we have the following inversion formula for $\pi_2$-spherical transform. 
\begin{thm}\label{thm:main}
For $f\in C_c^\infty(\mathfrak{a})^W$, we have
\begin{align*}
f(H)  =  \frac{1}{12} & \int_{\sqrt{-1}\mathfrak{a}^*}
  f^\wedge(\lambda)\varUpsilon^{\pi_2}
  (\phi^{\pi_2}_\lambda)(H)|c^{\pi_2}(\lambda)|^{-2}d\lambda \\
  &
  -\frac{c_1}{4\pi\sqrt{-1}}\int_{\sqrt{-1}\mathbb{R}}
(f^\wedge(\lambda)\varUpsilon^{\pi_2}(\phi^{\pi_2}_\lambda(H)))
|_{\lambda_{2\alpha_1+\alpha_2}=-\frac12}\,
p(\lambda_{\alpha_2})d\lambda_{\alpha_2}
\end{align*}
for $H\in \mathfrak{a}$, 
where 
\begin{align*}
p(\lambda_{\alpha_2}) & =
\Res_{
\lambda_{2\alpha_1+\alpha_2}=-\frac12}
(c^{\pi_2}(\lambda)^{-1}c^{\pi_2}(-\lambda)^{-1})
=\frac{\pi(36\lambda_{\alpha_2}^2-1)
(4\lambda_{\alpha_2}^3-\lambda_{\alpha_2})\sin\pi\lambda_{\alpha_2}}
 {2^{17}\cos\pi\lambda_{\alpha_2}}\\
& =-2^{-17}\pi\Bigl(36\left(\tfrac{\lambda_{\alpha_2}}{\sqrt{-1}}\right)^2+1\Bigr)
\Bigl(4\left(\tfrac{\lambda_{\alpha_2}}{\sqrt{-1}}\right)^2+1\Bigr)
\left(\tfrac{\lambda_{\alpha_2}}{\sqrt{-1}}\right)
\tanh \pi\left(\tfrac{\lambda_{\alpha_2}}{\sqrt{-1}}\right)
.
\end{align*}
\end{thm}

The Plancherel formula follows from Theorem~\ref{thm:main} by a 
standard argument as in the case 
of $\pi=\text{triv}$ (cf. the proof of \cite[Theorem~6.4.2]{GV} and 
\cite[Ch~IV Theorem~7.5]{Hel2}). 

\begin{cor}\label{cor:main}
For $f\in C_c^\infty(\mathfrak{a})^W$, we have
\begin{align*}
\frac{1}{12}
\int_\mathfrak{a}|f(H)|^2\delta_{G/K} & (H)dH  =  
\frac{1}{12}  \int_{\sqrt{-1}\mathfrak{a}^*}
  |f^\wedge(\lambda)|^2|c^{\pi_2}(\lambda)|^{-2}d\lambda \\
 - & \frac{c_1}{4\pi\sqrt{-1}}\int_{\sqrt{-1}\mathbb{R}}
|f^\wedge(\lambda)|_{\lambda_{2\alpha_1+\alpha_2}=-\frac12}|^2\,
p(\lambda_{\alpha_2})d\lambda_{\alpha_2}
\end{align*}
\end{cor}

As we see in \S~\ref{subsec:g2prem}, no discrete spectrum appears in 
the inversion formula and the Plancherel formula. 
In addition to the most continuous spectrum, there 
is a contribution of a principal series representation associated with a 
maximal parabolic subgroup whose Levi part corresponds to a short restricted 
root. 

\section*{Acknowledgement}
The first author was supported by JSPS KAKENHI Grant Number 18K03346. 
The authors 
 thank anonymous reviewers for careful reading our manuscript and for giving useful comments.

\medskip

\end{document}